\documentclass[12pt]{amsart}
\usepackage[T1]{fontenc}

\usepackage{amsmath}
\usepackage{amsfonts}
\usepackage{amssymb}
\usepackage{amsthm}
\usepackage{url}

\newcommand{\R}{\mathbb{R}}
\newcommand{\dif}[0]{\ensuremath{\,\mathrm{d}}}
\newcommand{\norm}[1]{\ensuremath{\Vert #1 \Vert}}

\newcommand{\inprod}[1]{\ensuremath{\langle #1 \rangle}}

\newcommand{\abs}[1]{\ensuremath{\vert #1 \vert}}
\newcommand{\scabs}[1]{\ensuremath{\left\vert #1 \right\vert}}
\newcommand{\A}[0]{\ensuremath{\mathcal{A}}}

\renewcommand{\S}[0]{\ensuremath{\mathcal{S}}}
\renewcommand{\L}[0]{\ensuremath{\mathcal{L}}}

\newcommand{\loc}[0]{\ensuremath{\mathrm{loc}}}

\DeclareMathOperator*{\dive}{div}

\def\vint_#1{\mathchoice%
          {\mathop{\kern 0.2em\vrule width 0.6em height 0.69678ex depth -0.58065ex
                  \kern -0.8em \intop}\nolimits_{\kern -0.4em#1}}%
          {\mathop{\kern 0.1em\vrule width 0.5em height 0.69678ex depth -0.60387ex
                  \kern -0.6em \intop}\nolimits_{#1}}%
          {\mathop{\kern 0.1em\vrule width 0.5em height 0.69678ex depth -0.60387ex
                  \kern -0.6em \intop}\nolimits_{#1}}%
          {\mathop{\kern 0.1em\vrule width 0.5em height 0.69678ex depth -0.60387ex
                  \kern -0.6em \intop}\nolimits_{#1}}}

\theoremstyle{plain}
\newtheorem{theorem}{Theorem}
\newtheorem{lemma}[theorem]{Lemma}

\numberwithin{theorem}{section}
\numberwithin{equation}{section}

\theoremstyle{definition}
\newtheorem{definition}[theorem]{Definition}

\newcommand{\Om}{\Omega}
\newcommand{\vp}{\varphi}

\newcommand{\ud}{\, d}

\title[A local approximation result]{Local approximation of superharmonic and superparabolic functions in nonlinear potential theory}

\author{Juha Kinnunen, Teemu Lukkari, Mikko Parviainen}
\address[Juha Kinnunen]{Department of Mathematics, Aalto University, P.O. Box 11100,
FI-00076 Aalto, Finland}
\email{juha.k.kinnunen@aalto.fi}
\address[Teemu Lukkari]{Department of Mathematics and Statistics, P.O. Box 35 (MaD),
FI-40014 University of Jyv\"askyl\"a, Finland}
\email{teemu.lukkari@jyu.fi}
\address[Mikko Parviainen]{Department of Mathematics and Statistics, P.O. Box 35 (MaD),
FI-40014 University of Jyv\"askyl\"a, Finland}
\email{mikko.j.parviainen@jyu.fi}



\dedicatory{Dedicated to Professor Bogdan Bojarski on the occasion of his 80th birthday}

\begin{document}

\begin{abstract}
  We prove that arbitrary superharmonic functions and
  superparabolic functions related to the $p$-Laplace and the $p$-parabolic equations
  are locally obtained as limits of supersolutions with desired convergence properties
  of the corresponding Riesz measures.
  As an application we show that a family of uniformly
  bounded supersolutions to the $p$-parabolic equation contains a
  subsequence that converges to a supersolution.
\end{abstract}

\maketitle

\section{Introduction}

This work gives approximation results related to the stationary $p$-Laplace equation
\begin{equation}\label{eq:p-laplace}
  \Delta_pu=\dive(\abs{\nabla u}^{p-2}\nabla u)=0
\end{equation}
and the time dependent  $p$-parabolic equation
\begin{equation}\label{eq:p-parabolic}
  \partial_tu-\Delta_pu=0.
\end{equation}
The results and arguments extend to a more general class of equations as well,
but we restrict our attention to the prototype equations for simplicity.
We have also deliberately decided to exclude the singular case $p<2$ from
our exposition.
The solutions of these equations form a similar basis for a nonlinear
potential theory as the solutions
of the Laplace and heat equations do in the classical theory.
Equation \eqref{eq:p-parabolic} is also known as the evolutionary $p$-Laplace
equation and the non-Newtonian filtration equation
in the literature.
For the elliptic regularity regularity theory we refer to \cite{BI}, \cite{HKM} and parabolic to \cite{DiB}.

In the nonlinear potential theory, so-called $p$-superharmonic and $p$-superparabolic
functions are essential, see  \cite{HKM},  \cite{KL} and \cite{LindqvistCrelle}.
They are defined as lower semicontinuous functions obeying
the comparison principle with respect to continuous solutions of the corresponding equation.
The main focus of this work is in the parabolic theory, but some observations may be
of interest already in the elliptic case.
In their definition, the superparabolic
functions are not required to have any derivatives,
and, consequently, it is not evident how to directly relate them to the equation.
However, by \cite{KinnunenLindqvist2} a $p$-superparabolic function has
spatial Sobolev derivatives with sharp local integrability bounds.
See also \cite{BG},  \cite{BDGO} and \cite{KilpelainenMaly1}.
Using this result we can show that every
superparabolic function $u$ satisfies the equation with measure data
\[
  \partial_t u-\dive(\abs{\nabla u}^{p-2}\nabla u)=\mu,
\]
where $\mu$ is the Riesz measure of $u$.  A rather delicate, yet relevant, point here
is that the spatial gradient of a superparabolic function is not
locally integrable to the natural exponent $p$.  Consequently, the
Riesz measure does not belong to the dual of the natural parabolic
Sobolev space.  For example, Dirac's delta is the Riesz measure for
the Barenblatt solution of the $p$-parabolic equation.

Since the weak supersolutions belong to the natural Sobolev, they
constitute a more tractable class of functions.
If we have an increasing sequence of
continuous supersolutions and the limit function is finite in
dense subset, then the limit function is superparabolic.
Moreover, if the limit function is bounded or belongs to the correct parabolic Sobolev space,
then it is a supersolution.
In fact, this gives a characterization of superparabolic functions.
The approximating sequence of supersolutions is usually constructed
through obstacle problems, see \cite{HKM} and \cite{LindqvistCrelle}
for the elliptic case and \cite{KinnunenLindqvist2} for the parabolic counterpart.
On the other hand, the standard approach in existence theory is to
approximate the Riesz measure of a superparabolic function
with smooth functions as in \cite{BG}, \cite{BDGO}, \cite{D} and \cite{KLP}.

Our argument is a combination of these procedures.
More precisely, we show that for a given superparabolic function, there
is a sequence of supersolutions, such that the Riesz measures are smooth
functions, which converges to the original function both in elliptic and
parabolic cases. This also provides a useful tool for existence results, because
our method gives a converging subsequence for any bounded family of supersolutions.
This is needed, for example, the proof of Theorem 4.3 in \cite{KLP}, and we
provide a detailed argument here.
In Appendix, we also provide a useful technical result related to the approximation of a
measure in the dual of the Sobolev space. This result may be of independent interest.

\section{The $p$-Laplace equation}

We will assume throughout the paper that $p>2$.
As usual, $W^{1,p}(\Omega)$ denotes the Sobolev space of functions in
$L^p(\Omega)$, whose distributional gradient belongs to $L^p(\Omega)$.
The space $W^{1,p}(\Omega)$ is equipped with the norm
\[
\|u\|_{W^{1,p}(\Omega)}
=\|u\|_{L^p(\Omega)}+\|\nabla
      u\|_{L^p(\Omega)}.
\]
The Sobolev space with zero boundary values, denoted by
$W^{1,p}_0(\Omega)$, is a completion of $C^{\infty}_0(\Omega)$ with
respect to the norm of $W^{1,p}(\Omega)$.

\begin{definition}
Let $\Omega$ be a bounded domain in $\R^n$.  A function $u\in
W^{1,p}_{\loc}(\Omega)$ is a weak solution of \eqref{eq:p-laplace}, if
\begin{displaymath}
  \int_{\Omega}\abs{\nabla u}^{p-2}\nabla u\cdot\nabla \varphi\dif x=0
\end{displaymath}
for all $ \varphi\in C^\infty_0(\Omega)$. Further, $u$ is a weak
supersolution if the integral is nonnegative for all nonnegative
test functions $\varphi$.
\end{definition}

By elliptic regularity theory, weak solutions  are locally H\"older continuous after a possible redefinition on a set of measure
zero and weak supersolutions are lower semicontinuous with the same interpretation, see for example \cite{HKM}.

Our argument also applies to more general equations of the type
\[
\dive\A(x,\nabla u)=0
\]
with the structural conditions
\begin{enumerate}
\item $x\mapsto\A(x,\xi)$ is measurable for all $\xi\in\R^n$,
\item $\xi\mapsto\A(x,\xi)$ is continuous for almost all
  $x\in\Omega$,
\item $\A(x,\xi)\cdot\xi\geq\alpha \abs{\xi}^{p}$, $\alpha>0$, for almost all  $x\in\Omega$ and all $\xi\in\R^n$,
\item $\abs{\A(x,\xi)}\leq\beta\abs{\xi}^{p-1}$ for almost all
  $x\in\Omega$ and all $\xi\in\R^n$,
\item $(\A(x,\xi)-\A(x,\eta))\cdot(\xi-\eta)>0$ for almost all
  $x\in\Omega$ and all $\xi, \eta\in \R^n$ with
  $\xi\not=\eta$.
\end{enumerate}
However, for simplicity, we only consider the prototype $p$-Laplace equation in this work.

\begin{definition}
A lower semicontinuous function $u: \Omega\to (-\infty,\infty]$ is $p$-super\-har\-mo\-nic,
if $u$ is finite in a dense subset of $\Omega$, and the following comparison principle holds in
every $\Omega'\Subset\Omega$:
If $h\in C(\overline{\Omega'})$ is a weak solution to \eqref{eq:p-laplace} in $\Omega'\Subset\Omega$
and $u \geq  h$ on $\partial \Omega'$, then $u\geq  h$ in $\Omega'$.
\end{definition}

Clearly, $\min(u,v)$ and $\alpha u+\beta$ are $p$-superharmonic if $u$ and $v$ are and
$\alpha,\beta\in\mathbb R$ with $\alpha\ge0$.
We refer to \cite{HKM} and \cite{LindqvistCrelle} for more on $p$-superharmonic functions.

There is a relation between weak supersolutions and
$p$-superharmonic functions.
Weak supersolutions satisfy the comparison principle
and, roughly speaking, they are $p$-superharmonic,
provided the issue about lower semicontinuity is properly handled:
every supersolution has a lower semicontinuous representative.
In particular, a lower semicontinuous supersolution is $p$-superharmonic.
However, not all $p$-superharmonic functions are weak supersolutions, as the example
\[
u(x)=
\begin{cases}
|x|^{(p-n)/(p-1)}&\quad\text{if}\quad p\ne n,\\
\log|x|&\quad\text{if}\quad p=n,
\end{cases}
\]
shows. This function is $p$-superharmonic in $\mathbb R^n$, but it fails to belong to $W^{1,p}_{\loc}(\mathbb R^n)$,
when $1<p<n$. However, the function $u$ is a weak solution in $\mathbb R^n\setminus\{0\}$ and it is a solution of
\[
-\Delta_p u=\delta,
\]
where the right-hand side is Dirac's delta at the origin.

Next we recall the obstacle problem in the calculus of variations.

\begin{definition}
For $\psi\in C(\Omega)\cap W^{1,p}(\Omega)$,
consider the class $\mathcal F_\psi$ of functions $v\in C(\Omega)\cap W^{1,p}(\Omega)$ such that $v-\psi\in W^{1,p}_0(\Omega)$
and $v\ge\psi$ in $\Omega$.
The problem is to find a function $u\in\mathcal F_\psi$ such that
\[
\int_{\Omega}\abs{\nabla u}^{p-2}\nabla u\cdot(\nabla v-\nabla u)\dif x\ge0
\]
for all $v\in\mathcal F_\psi$. The function $\psi$ is called the obstacle and $u$ the solution
to the obstacle problem in $\Omega$ with the obstacle $\psi$.
\end{definition}

By \cite{HKM}, there is a unique solution of the obstacle problem provided $\mathcal F_\psi$
is not empty, because $\mathcal F_\psi$ is convex.
Continuity follows from the standard regularity theory, see \cite{HKM}.

For every $p$-superharmonic function $u$ and $\Omega'\Subset\Omega$,
there is an increasing sequence $u_i$, $i=1,2,\dots$, of continuous weak supersolutions in $\Omega'$,
converging to $u$. Such a sequence is found by means of the obstacle
problem as follows: By semicontinuity, there is an increasing sequence
$\psi_i$, $i=1,2,\ldots$, of smooth functions converging to $u$
pointwise. Let $u_i$ be the solution to the obstacle problem with
obstacle $\psi_i$. Then it can be shown that
the sequence $(u_i)$ has the desired properties, see \cite{HKM} and \cite{LindqvistCrelle}.

An important consequence of this approximation is a local summability result for
$p$-superharmonic functions. This in turn implies that the solutions
to the obstacle problem converge in a Sobolev space, and the existence
of the Riesz measures of $p$-superharmonic functions. We record these
facts in the following theorem. The example above shows that the bounds
are sharp.

\begin{theorem}
  Let $u$ be a $p$-superharmonic function in $\Omega$ and let $u_i$, $i=1,2,\dots$, be a sequence
  of solutions of the obstacle problem as above. Then
  \begin{enumerate}
  \item $u\in L^{r}_{\loc}(\Omega)$ and $\abs{\nabla u}\in
    L^q_{\loc}(\Omega)$ for every $r$ and $q$ such that
    \begin{displaymath}
      1\le r<\frac{n}{n-p}(p-1) \quad\text{and}\quad 1\le q<\frac{n}{n-1}(p-1).
    \end{displaymath}
  \item There is a nonnegative Radon measure $\mu$ such that $-\Delta_p u=\mu$.
  \item The sequence $u_i$, $i=1,2,\dots$, converges to $u$ in $W^{1,q}_{\loc}(\Omega)$ as $i\to\infty$.
  \end{enumerate}
\end{theorem}

The equation $-\Delta_p u=\mu$ means that
\[
 \int_{\Omega}\abs{\nabla u}^{p-2}\nabla u\cdot\nabla \varphi\dif x=\int_\Omega\varphi\dif\mu
\]
for all $ \varphi\in C^\infty_0(\Omega)$. Here $\mu$ is the Riesz measure of $u$.
Observe, that for $p<n$, the gradient of a $p$-superharmonic function $u$ does not belong
to $L^p_{\loc}(\Omega)$, but however, it is a weak solution of the above measure data problem.
If, in addition, $u\in W^{1,p}_{\loc}(\Omega)$, then the restriction of $\mu$ belongs to $W^{-1,p'}(\Omega')$,
which is the dual of the space $W^{1,p}_0(\Omega')$, for every $\Omega'\Subset\Omega$.

For later use, we record the following lemma. The proof is based on
a characterization of Sobolev spaces and their duals by means of
Bessel potentials, and is given in Appendix \ref{ap:dual-approx}.

\begin{lemma}\label{lem:meas-approx}
  Let $\mu \in W^{-1,p'}(\Om)$ be a positive measure. Then
  there is a sequence $f_i$, $i=1,2,\dots$,  of positive smooth functions such that
$f_i\to \mu$ in $W^{-1,p'}(\Om)$ as $i\to \infty$.
\end{lemma}

Now we are ready to prove an approximation lemma for weak supersolutions.

\begin{lemma}\label{lem:obstacle-approx}
  Let $u$ be a continuous weak supersolution in $\Omega$, $\Omega'\Subset\Omega$ an open set
  and let $\mu$ be the Riesz measure of $u$.
  Then there is a sequence $u_i$, $i=1,2,\dots$,  of solutions to $-\Delta_p u_i=f_i$,
  where the functions $f_i$ are nonnegative and smooth with the
  following properties:
  \begin{enumerate}
  \item $u_i\to u$ in $W^{1,p}(\Omega')$, and pointwise almost everywhere in $\Omega'$,
  \item $\nabla u_i\to\nabla u$ pointwise almost everywhere in $\Omega'$ and
  \item  $f_i\to\mu$ in $W^{-1,p'}(\Omega')$ as $i\to\infty$.
 \end{enumerate}
 \end{lemma}

\begin{proof}
  Since $\mu$ is the Riesz measure of $u$, we have
  \begin{displaymath}
\begin{split}
|\inprod{\mu,\varphi}|
&=\left|\int_{\Omega'}\varphi\dif \mu\right|
=\left|\int_{\Omega'} \abs{\nabla u}^{p-2}\nabla u\cdot \nabla \varphi\dif x\right|\\
&\le\|\nabla u\|_{L^p(\Omega')}^{p-1}\|\nabla\varphi\|_{L^p(\Omega')}
\end{split}
  \end{displaymath}
for every $\varphi\in C_0^\infty(\Omega')$.
Because $u\in W^{1,p}(\Omega')$, we conclude that $\mu\in W^{-1,p'}(\Omega')$.
By Lemma~\ref{lem:meas-approx}, choose  a
  sequence $(f_i)$ of smooth, nonnegative functions converging to
  $\mu$ in $W^{-1,p'}(\Omega')$, and let $u_i$, $i=1,2,\dots$, be the unique solutions of the boundary value problems
  \begin{displaymath}
    \begin{cases}
      -\Delta_p u_i=f_i \quad \text{in}\quad \Omega',\\
      u_i-u\in W_0^{1,p}(\Omega').
    \end{cases}
  \end{displaymath}
Using $u-u_i$ as a test function in the equations for $u$ and $u_i$ and
subtracting the obtained equations, we have
\[
  \begin{split}
  \int_{\Omega'}(\abs{\nabla u}^{p-2}\nabla u-\abs{\nabla u_i}^{p-2}\nabla u_i)\cdot
    \nabla (u-u_i)\dif x=\inprod{\mu-f_i,u-u_i},
  \end{split}
 \]
 where $\inprod{\cdot, \cdot}$ denotes the duality pairing between
  $W^{1,p}_0(\Omega')$ and $W^{-1,p'}(\Omega')$. The left-hand side is positive,
  so we may take absolute values on the right-hand side.
 By an elementary inequality for vectors and the Sobolev inequality, we obtain
\begin{equation}
\label{eq:startup-estimate}
  \begin{split}
  \int_{\Omega'}&\abs{\nabla (u-u_i)}^p\dif x\\
  &\le  \int_{\Omega'}(\abs{\nabla u}^{p-2}\nabla u-\abs{\nabla u_i}^{p-2}\nabla u_i)\cdot
    \nabla (u-u_i)\dif x\\
   &=|\inprod{\mu-f_i,u-u_i}|
    \leq\norm{\mu-f_i}_{W^{-1,p'}(\Omega')}\norm{u-u_i}_{W^{1,p}_0(\Omega')}\\
    &\leq c\norm{\mu-f_i}_{W^{-1,p'}(\Omega')}\norm{\nabla(u-u_i)}_{L^p(\Omega')},
  \end{split}
\end{equation}
where $c=c(n,p)$.
  From this, it follows that
  \begin{displaymath}
    \norm{\nabla (u-u_i)}_{L^p(\Omega')}^{p-1}
    \leq c\norm{\mu-f_i}_{W^{-1,p'}(\Omega')}.
  \end{displaymath}
 This implies that $u_i\to u$ in $W^{1,p}(\Omega')$ as $i\to\infty$.
  \end{proof}

As a consequence of the previous result, we obtain a corresponding approximation
result for $p$-superharmonic functions.

\begin{theorem} \label{thm:p-superapprox}
  Let $u$ be a $p$-superharmonic function in $\Omega$, $\Omega'\Subset\Omega$ an open set
  and let $\mu$ be the Riesz measure of $u$.
  Then there is a sequence $u_i$, $i=1,2,\dots$,  of solutions to $-\Delta_p u_i=f_i$ in $\Omega'$,
  where the functions $f_i$ are nonnegative and smooth with the
  following properties:
  \begin{enumerate}
  \item $u_i\to u$ in $W^{1,q}(\Omega')$ with $1\le q<n(p-1)/(n-1)$,
    and pointwise almost everywhere in $ \Omega' $,
  \item  $\nabla u_i\to\nabla u$  pointwise  almost everywhere in $\Omega'$ and
  \item  the functions $f_i$ converge weakly to $\mu$ in the sense of measures as $i\to\infty$.
  \end{enumerate}
\end{theorem}

\begin{proof}
Let $\widetilde{u}_i$, $i=1,2,\dots$, be continuous solutions to the obstacle problem converging to $u$ in $W^{1,q}(\Omega')$.
For each $i=1,2,\ldots$, we may apply Lemma \ref{lem:obstacle-approx} to find a solution of $-\Delta_p u_i=f_i$
in $\Omega'$ such that
\[
\norm{u_i-\widetilde{u}_i}_{W^{1,p}(\Omega')} \leq\frac1i.
\]

  Since
  \begin{displaymath}
    \norm{u-u_i}_{W^{1,q}(\Omega')}\leq
    \norm{u-\widetilde{u}_i}_{W^{1,q}(\Omega')}
    +\norm{\widetilde{u}_i -u_i}_{W^{1,q}(\Omega')},
  \end{displaymath}
  we see that $(u_i)$ converges to $u$ in $W^{1,q}(\Omega')$ as $i\to\infty$.
  By passing to a subsequence, if necessary, we may assume that $\nabla u_i\to \nabla u$ pointwise
  almost everywhere in $\Omega'$. Recall that $q<n(p-1)/(n-1)$; thus the sequence
  $(\abs{\nabla u_i}^{p-2}\nabla u_i)$ is bounded in $L^r(\Omega')$ for
  $1<r<n/(n-1)$. In addition, the limit of the weakly convergent
  subsequence is $\abs{\nabla u}^{p-2}\nabla u$
  by the pointwise convergence.  This allows us to conclude that
  \begin{align*}
    \lim_{i\to \infty}\int_{\Omega'}\varphi f_i\dif x=&\lim_{i\to \infty}
    \int_{\Omega'}\abs{\nabla u_i}^{p-2}\nabla
    u_i\cdot\nabla \varphi\dif x\\
    =&\int_{\Omega'} \abs{\nabla u}^{p-2}\nabla
    u\cdot\nabla \varphi\dif x\\
    =&\int_{\Omega'}\varphi \dif \mu
  \end{align*}
 for every $\varphi\in C^\infty_0(\Omega')$, and the proof is complete.
\end{proof}

\section{The $p$-parabolic equation}

In order to discuss the parabolic case, some preparations are needed.
Let $\Omega$ be an open and bounded set in $\R^n$ with $n\geq 1$.
We denote
\[
\Omega_T=\Omega\times(0,T)
,\] where $0<T<\infty$.
For an open set $U$ in $\R^n$ we write
\[
\Omega_{t_1,t_2}=\Omega\times(t_1,t_2),
\]
 where $0<t_1<t_2<\infty$.
The parabolic boundary of $\Omega_{t_1,t_2}$ is
\begin{displaymath}
  \partial_p\Omega_{t_1,t_2}=\big(\partial \Omega\times[t_1,t_2]\big)\cup
  (\overline{\Omega}\times\{t_1\}).
\end{displaymath}

The parabolic Sobolev space $L^p(0,T;W^{1,p}(\Omega))$ consists of
measurable functions $u:\Omega_T\to[-\infty,\infty]$ such that for
almost every $t\in (0,T)$, the function $x\mapsto u(x,t)$ belongs to
$W^{1,p}(\Omega)$ with the norm
\begin{equation}\label{eqn:parabsobo}
\|u\|_{L^p(0,T;W^{1,p}(\Omega))}
=\left(\int_{\Omega_T}(|u|^p+|\nabla u|^p)\dif x\dif t\right)^{1/p}<\infty.
\end{equation}
A function $u\in L^p(0,T;W^{1,p}(\Omega))$ belongs to the space
$L^{p}(0,T;W^{1,p}_0(\Omega))$ if $x\mapsto u(x,t)$ belongs to
$W^{1,p}_0(\Omega)$ for almost every $t\in (0,T)$.  The local space
$L^p_{\loc}(0,T;W^{1,p}_{\loc}(\Omega))$ consists of functions that
belong to the parabolic Sobolev space in every
$\Omega'_{t_1,t_2}\Subset\Omega_T$.

\begin{definition}
A function $u\in L^p_{\loc}(0,T;W^{1,p}_{\loc}(\Omega))$ is a weak solution
of \eqref{eq:p-parabolic} in $\Omega_T$, if
  \begin{equation}\label{eq:weak-form}
    -\int_{\Omega_T}u\frac{\partial \varphi}{\partial t}\dif x\dif t
    +\int_{\Omega_T}\abs{\nabla u}^{p-2}\nabla u\cdot\nabla \varphi\dif x\dif t=0
  \end{equation}
  for all test functions $\varphi\in C^\infty_0(\Omega_T)$.  The
  function $u$ is a supersolution if the integral in
  \eqref{eq:weak-form} is nonnegative for nonnegative test functions.
  In a general open set $V$ of $\R^{n+1}$, the above notions are to be
  understood in a local sense, i.e. $u$ is a solution if it is a
  solution in every set $\Omega_{t_2,t_2}\Subset V$.
\end{definition}

By parabolic regularity theory, the solutions
are locally H\"older continuous after a possible redefinition on a set of measure
zero, see \cite{DiB}.
In general, the time derivative $u_t$ does not exist as a function.
This is a principal, well-recognized difficulty with the definition.
Namely, in proving estimates, we usually need a test function
$\varphi$ that depends on the solution itself, for example
$\varphi=u\zeta$ where $\zeta$ is a smooth cutoff function.
Then we cannot avoid that the forbidden quantity
$u_t$ shows up in the calculation of $\varphi_t$.
In most cases, we can easily overcome this difficulty by using
an equivalent definition in terms of Steklov averages, as on pages 18 and 25 of \cite{DiB}.
Alternatively, we can proceed using convolutions with smooth
mollifiers.

\begin{definition}
A lower semicontinuous function $u:{\Omega_T}\to(-\infty,\infty]$ is
$p$-super\-para\-bolic in ${\Omega_T}$, if
 $u$ is finite in a dense subset of $\Omega_T$, and the following parabolic comparison principle holds:
If $h\in C(\overline{\Omega'}_{t_1,t_2})$ is a solution of \eqref{eq:p-parabolic} in $\Omega'_{t_1,t_2}\Subset{\Omega_T}$
and $h\leq u$ on the parabolic boundary $\partial_p\Omega'_{t_1,t_2}$, then $h\leq u$ in $\Omega_{t_1,t_2}$.
\end{definition}

It follows immediately from the definition that, if $u$ and
$v$ are $p$-superparabolic functions, so are
their pointwise minimum $\min(u,v)$ and
$u+\beta$, $\beta\in\mathbb R$.
Observe that $u+v$ and
$\alpha u$, with $\alpha\ge0$, are not superparabolic in general.
In addition, the class of superparabolic functions is closed
with respect to to the increasing convergence, provided the limit
function is finite in a dense subset.
We refer to \cite{KL} for more information about $p$-superparabolic functions.

A lower semicontinuous representative of a weak supersolution is $p$-superparabolic, see \cite{Ku}, but
as in the elliptic case, not all $p$-super\-para\-bo\-lic functions are weak supersolutions.
For example, consider the  Barenblatt solution $\mathcal B_p:\mathbb R^{n+1}\to[0,\infty)$,
\[
\mathcal B_p(x,t)
=\begin{cases}
t^{-n/\lambda}
\bigg(c-\dfrac{p-2}p\lambda^{1/(1-p)}
\bigg(\dfrac{|x|}{t^{1/\lambda}}\bigg)^{p/(p-1)}\bigg)_+^{(p-1)/(p-2)},&
\,t>0,
\\
0,&\, t\le 0,
\end{cases}
\]
where $\lambda=n(p-2)+p$, $p>2$, and the constant $c$ is
usually chosen so that
\[
\int_{\mathbb R^n}\mathcal B_p(x,t)\dif x=1
\]
for every $t>0$.
The Barenblatt solution is a weak solution of \eqref{eq:p-parabolic} in the upper half space
\[
\{(x,t)\in\mathbb R^{n+1}:x\in\mathbb R^n,\,t>0\}
\]
and it is a very weak solution of the equation
\[
\partial_t\mathcal B_p
-\Delta_p\mathcal B_p=\delta
\]
in $\mathbb R^{n+1}$, where the right-hand side is Dirac's delta at the origin.
In contrast with the heat kernel, which is strictly positive,
the Barenblatt solution has a bounded support
at a given instance $t>0$.
Hence the disturbances propagate with finite speed when $p>2$.
The Barenblatt solution is $p$-superparabolic in $\mathbb R^n$, but it is not a supersolution in an open set that contains the origin,
because it does not belong to $L^p_{\loc}(0,T;W^{1,p}_{\loc}(\Omega))$.

The obstacle problem in the calculus of variations is a basic tool in the study of the
$p$-superparabolic functions as in the elliptic case.

\begin{definition}
Let $\psi\in C^\infty(\mathbb R^{n+1})$ and consider the class $\mathcal F_\psi$ of all functions
$v\in C(\overline \Omega_T)$ such that $v\in L^p(0,T;W^{1,p}(\Omega))$, $v=\psi$ on the parabolic boundary
of $\Omega_T$ and $v\ge\psi$ in $\Omega_T$.
The problem is to find $u\in\mathcal F_\psi$ such that
\[
\begin{split}
\int_0^T\int_\Omega&\Big(|\nabla u|^{p-2}\nabla u\cdot(\nabla v-\nabla u)+(v-u)\frac{\partial v}{\partial t}\Big)\dif x\dif t
\\
&\ge\frac12\int_\Omega |v(x,T)-u(x,T)|^2\dif x
\end{split}
\]
for all smooth functions $v$ in the class $\mathcal F_\psi$.
In particular, $u$ is a continuous supersolution of \eqref{eq:p-parabolic}.
Moreover, in the open set $\{u>\psi\}$ the function $u$ is a solution of \eqref{eq:p-parabolic}.
\end{definition}

For existence results for smooth sets $\Omega$, see \cite{AL}, \cite{BDMObstacle} and \cite{KKS}.
Characterizations of the solutions of the parabolic obstacle problem have also been studied in \cite{LiPa}.

As before, we may construct continuous supersolutions converging to a
given $p$-superparabolic function in a space-time cylinder
$\Omega_{t_1,t_2}$ by employing the obstacle problem. For details, see
\cite{KinnunenLindqvist2} and \cite{KKP}. Below, we will also need the fact that the
solutions to the obstacle problem have a time derivative in $L^{p'}(0,T; W^{-1,p'}(\Omega))$,
which is the dual space of $L^p(t_1,t_2;W_0^{1,p}(\Omega))$. The latter fact is contained in the
existence result proved in \cite{BDMObstacle}. Here we use the fact
that the obstacle functions are smooth. The following summability estimates are
proved in \cite{KinnunenLindqvist1}.
See also \cite{BG} and \cite{BDGO}.

\begin{theorem}
  Let $u$ be a $p$-superparabolic function in $\Omega_T$,
    $\Omega'_{t_1,t_2}\Subset\Omega_T$ be an open set and let $u_i$, $i=1,2,\dots$,
  be a sequence of solutions of the obstacle problem in $\Omega'_{t_1,t_2}$ as above.
  Then
  \begin{enumerate}
  \item $u\in L^{r}_{\loc}(\Omega_T)$ and $\abs{\nabla u}\in
    L^q_{\loc}(\Omega_T)$ for any $r$ and $q$ such that
    \begin{displaymath}
    1\le  r<p-1+p/n \quad\text{and}\quad 1\le q<p-1+1/(n+1),
    \end{displaymath}
  \item there is a positive measure $\mu$ such that
    \begin{displaymath}
      \partial _t u-\Delta_p u=\mu
    \end{displaymath}
    in $\Omega_T$ and
  \item the sequence $u_i$, $i=1,2,\dots$, converges in $L^q(t_1,t_2;W^{1,q}(\Omega'))$ as $i\to\infty$.
  \end{enumerate}
\end{theorem}

The Barenblatt solution shows that these critical integrability exponents
for a $p$-superparabolic function and its gradient are optimal.

\begin{theorem}\label{thm:p-superpara-approx}
   Assume that $u$ is a $p$-superparabolic function in $\Omega_T$, let
    $\Omega'_{t_1,t_2}\Subset\Omega_T$ be an open set and let $\mu$ be the Riesz measure of $u$.
 Then there is a sequence $u_i$, $i=1,2,\dots$,  of solutions to
  \begin{displaymath}
    \partial_t u_i-\Delta_p u_i=f_i
    \end{displaymath}
  in $\Omega'_{t_1,t_2}$, where the functions $f_i$ are nonnegative and smooth, with the
  following properties:
 \begin{enumerate}
  \item $u_i\to u$ in $L^q(t_1,t_2;W^{1,q}(\Omega'))$ with $1\le q<p-1+1/(n+1)$
    and pointwise almost everywhere in $\Omega'_{t_1,t_2}$,
  \item $\nabla u_i\to\nabla u$
  pointwise almost everywhere in $\Omega'_{t_1,t_2}$,
  \item $\partial_t u_i\in L^{p'}(t_1,t_2;W^{-1,p'}(\Omega'))$ for every $i=1,2,\dots$, and
  \item  $f_i\to \mu$ in $L^{p'}(t_1,t_2;W^{-1,p'}(\Omega'))$ as $i\to\infty$.
  \end{enumerate}
\end{theorem}

We use the following approximation result. As in the elliptic case,
the proof is based on Bessel potentials, and is found in Appendix \ref{ap:dual-approx}.

\begin{lemma}\label{lem:para-meas-approx}
  Let $\mu\in L^{p'}(0,T; W^{-1,p'}(\Omega))$ be a positive
  measure. Then there are smooth functions such
  that $f_i\to\mu$ in $L^{p'}(0,T; W^{-1,p'}(\Omega))$ as $i\to \infty$.
\end{lemma}

With the previous lemma, we can prove the following convergence result
for solutions of the obstacle problems.

\begin{lemma}\label{lem:paraobstacle-approx}
  The claim of Theorem \ref{thm:p-superpara-approx} holds when $u$ is
  a solution to the obstacle problem with a smooth obstacle.
  In addition,  the sequence converges
  in $L^p(t_1,t_2;W^{1,p}(\Omega'))$.
\end{lemma}

\begin{proof}
  Let $\mu$ be the Riesz measure of $u$, $\Omega'\Subset\Omega$ and $0<t_1<t_2<T$.
  Since the obstacle is smooth, we have
  \[
  \partial_t u\in L^{p'}(t_1,t_2;W^{-1,p'}(\Omega')),
  \]
 and consequently
  \begin{displaymath}
  \begin{split}
    |\inprod{\mu,\varphi}|
    &=\left|\int_{\Omega'_{t_1,t_2}}\varphi\dif\mu\right|\\
    &=\left| -\int_{\Omega'_{t_1,t_2}}u \partial_t\varphi\dif x\dif t
    +\int_{\Omega'_{t_1,t_2}}\abs{\nabla u}^{p-2}\nabla u\cdot\nabla \varphi\dif x\dif t\right|\\
    &\leq \norm{\partial_tu}_{L^{p'}(t_1,t_2;W^{-1,p'}(\Omega'))}
    \norm{\varphi}_{L^p(t_1,t_2;W_0^{1,p}(\Omega'))}
  \\
  &\qquad\qquad  +\norm{\nabla u}_{L^p(\Omega'_{t_1,t_2})}^{p-1}\norm{\nabla \varphi}_{L^p(\Omega'_{t_1,t_2})}
 \\
 &\le\left (\norm{\partial_tu}_{L^{p'}(t_1,t_2;W^{-1,p'}(\Omega'))}
  +\norm{\nabla u}_{L^p(\Omega'_{t_1,t_2})}^{p-1}\right)
  \norm{\varphi}_{L^p(t_1,t_2;W_0^{1,p}(\Omega'))}
  \end{split}
   \end{displaymath}
  for every $\varphi\in C_0^\infty(\Omega'_{t_1,t_2})$.
 This implies that
 \[
 \mu\in L^{p'}(t_1,t_2;W^{-1,p'}(\Omega'))
 \]
 with
  \begin{displaymath}
    \norm{\mu}_{L^{p'}(t_1,t_2;W^{-1,p'}(\Omega'))}\leq \norm{\partial_t
      u}_{L^{p'}(t_1,t_2;W^{-1,p'}(\Omega'))}+
    \norm{\nabla u}^{p-1}_{L^p(\Omega'_{t_1,t_2})}.
  \end{displaymath}
 We choose nonnegative functions $f_i\in C^\infty(\Omega'_{t_1,t_2})$, $i=1,2,\dots$, such that
  \begin{displaymath}
    f_i\to \mu  \quad \text{in} \quad
    L^{p'}(t_1,t_2;W^{-1,p'}(\Omega'))
  \end{displaymath}
 as $i\to\infty$.
 Let $u_i$, $i=1,2,\ldots$, be the unique solutions
  of the boundary value problems
  \begin{displaymath}
    \begin{cases}
      \partial_t u_i-\Delta_p u_i=f_i& \quad \text{in}\quad \Omega'_{t_1,t_2},\\
      u_i=u& \quad\text{on}\quad \partial_p \Omega'_{t_1,t_2}.
    \end{cases}
  \end{displaymath}
  Using $u-u_i$ as a test function in the equations for $u$ and $u_i$, and
  subtracting the obtained equations, we have
  \begin{equation}\label{eq:diffpara}
 \begin{split}
    &\inprod{\partial_t (u-u_i),u-u_i}\\
    &\qquad+\int_{\Omega'_{t_1,t_2}}(\abs{\nabla
      u}^{p-2}\nabla u-\abs{\nabla
      u_i}^{p-2}\nabla u_i)\cdot(\nabla u-\nabla u_i)\dif x\dif t\\
    &=\inprod{\mu-f_i,u-u_i}.
  \end{split}
  \end{equation}
  A formal computation gives
  \begin{displaymath}
  \begin{split}
    \inprod{\partial_t (u-u_i),u-u_i}
&=\int_{\Omega'_{t_1,t_2}}(u-u_i) \partial_t (u-u_i)\dif x\dif t\\
 &=\frac{1}{2}\int_{\Omega'}(u-u_i)^2(x,t_2)\dif x\geq 0.
 \end{split}
  \end{displaymath}
A rigorous argument can be based on Steklov averages, see \cite{DiB}.
  We conclude that
  \begin{equation}\label{eq:graddiff}
 \begin{split}
    \int_{\Omega'_{t_1,t_2}}&\abs{\nabla( u-u_i)}^p\dif x\dif t\leq
    \inprod{\mu-f_i,u-u_i}\\
    &\leq
    c\norm{\mu-f_i}_{L^{p'}(t_1,t_2;W^{-1,p'}(\Omega'))}\norm{u-u_i}_{L^p(t_1,t_2;W_0^{1,p}(\Omega'))}.
 \end{split}
  \end{equation}
 Here we also applied an elementary inequality for vectors.
 By the Sobolev inequality we have
 \[
  \norm{u-u_i}_{L^p(\Omega'_{t_1,t_2})}
 \le c\norm{\nabla(u-u_i)}_{L^p(\Omega'_{t_1,t_2})}
 \]
  with $c=c(n,p)$.
  From \eqref{eq:graddiff} and the Sobolev inequality above, we obtain
  \begin{displaymath}
    \norm{\nabla( u-u_i)}_{L^p(\Omega'_{t_1,t_2})}
    \leq c\norm{\mu-f_i}_{L^{p'}(t_1,t_2;W^{-1,p'}(\Omega'))}^{1/(p-1)}.
  \end{displaymath}
  The right-hand side tends to zero as $i\to\infty$.
  Thus we obtain
 \[
 u_i\to u\quad\text{in}\quad L^p(t_1,t_2;W^{1,p}(\Omega'))
 \]
 as $i\to\infty$.
 By passing to a subsequence, if necessary, we obtain the convergence almost everywhere.
\end{proof}

\begin{proof}[Proof of Theorem \ref{thm:p-superpara-approx}]
  The proof is now similar to that of Theorem \ref{thm:p-superapprox};
  let $\widetilde{u_i}$, $i=1,2,\dots$, be the solutions to the obstacle problem
  converging to $u$, and apply Lemma \ref{lem:paraobstacle-approx} to
  pick the functions $u_i$, $i=1,2,\dots$, such that
  \[
  \norm{\widetilde{u}_i-u_i}_{L^p(t_1,t_2;W^{1,p}(\Omega'))}\leq\frac1i.
  \]
  The convergence of $u_i$ to $u$ then follows by using the
  triangle inequality. For the weak convergence of the functions
  $f_i$, we note that $(\abs{\nabla u_i}^{p-2}\nabla u_i)$ is bounded in
  $L^r(\Omega'_{t_1,t_2})$ for some $r>1$, and by the pointwise convergence
  we have that the weak limit must be $\abs{\nabla u}^{p-2}\nabla u$;
  thus a computation similar to the one in the proof of
  Theorem \ref{thm:p-superapprox}  shows the weak convergence of  the sequence $(f_i)$ to $\mu$
  as $i\to\infty$.
\end{proof}

\section{A convergence result for weak supersolutions}

As an application of the above results, we prove a pointwise convergence theorem
for a bounded sequence of weak supersolutions of the $p$-parabolic equation in the
slow diffusion case $p>2$.
This kind of approximation is needed the proof of Theorem 4.3 in \cite{KLP}, and we
provide a detailed argument here.
The idea of the proof is to use Theorem \ref{thm:p-superpara-approx}
to choose approximations of the supersolutions $u_i$, apply a
compactness result to the approximations to obtain a limit function, and then show that also
the original functions converge to the same limit.
The advantage of this approach is that the Riesz measures of the approximations are known
to be functions, so that norms of their time derivatives can be estimated on time slices.

\begin{theorem}
  Let $u_i$, $i=1,2,\ldots$, be weak supersolutions of
  \eqref{eq:p-parabolic}  in $\Omega_T$ such that $\abs{u_i}\leq M<\infty$ for every $i=1,2,\dots$.
  Then, for a subsequence still denoted by $(u_i)$, there is a function $u$
  such that $u_i\to u$ and $\nabla u_i\to \nabla u$ pointwise almost  everywhere in $\Omega_T$
  as $i\to\infty$,  and, moreover, $u$ is a weak supersolution of \eqref{eq:p-parabolic}  in $\Omega_T$.
\end{theorem}

\begin{proof}
  Let $\mu_i$ be the Riesz measures associated to $u_i$, $i=1,2,\ldots$
  For $\Omega''_{s_1,s_2}\Subset\Omega'_{t_1,t_2}\Subset \Omega_T$, we have
  \begin{displaymath}
    \norm{\nabla u_i}_{L^p(\Omega'_{t_1,t_2})}\leq cM, \quad\text{and}
    \quad\mu_i(\Omega'_{t_1,t_2})\leq cM
  \end{displaymath}
 by Caccioppoli's inequality, see the proof of Theorem 4.3 in \cite{KLP}.
  By Theorem \ref{thm:p-superpara-approx}, for each $i$, we may choose  $v_i$ such
  that
  \begin{displaymath}
    \partial_t v_i-\Delta_p v_i=f_i
  \end{displaymath}
where $f_i$ is smooth with the properties
  \begin{displaymath}
    \norm{f_i}_{L^1(\Omega''_{s_1,s_2})}\leq 2cM, \quad\text{and}\quad
    \norm{u_i-v_i}_{L^q(t_1,t_2;W^{1,q}(\Omega'))}
    \leq \frac1i
  \end{displaymath}
  for some $q$ with $p-1<q<p$.

  We aim at applying a compactness result of \cite{Simon} to the sequence $(v_i)$.
  To this end, we need bounds for the gradients and the time derivatives
  of the functions $v_i$.
  We have
  \begin{align*}
    \norm{v_i}_{L^q(t_1,t_2;W^{1,q}(\Omega'))}\leq &\norm{ u_i}_{L^q(t_1,t_2;W^{1,q}(\Omega'))}
    +\norm{v_i-u_i}_{L^q(t_1,t_2;W^{1,q}(\Omega'))}\\
    \leq  & c(M+\norm{\nabla u_i}_{L^p(\Omega'_{t_1,t_2})})+\frac1i  \\
    \leq & cM+1.
  \end{align*}

Then we consider a bound for the time derivative.  From the equation satisfied by $v_i$ we see that
  \begin{equation}\label{eq:comppf1}
    \partial_t v_i\in L^1(\Omega''_{s_1,s_2})+L^r(s_1,s_2;W^{-1,r}(\Omega''))
  \end{equation}
  where $r=q/(p-1)>1$, and functions in $L^1(\Omega''_{s_1,s_2})$ are
  identified with distributions in the usual manner. We denote the functional in the
  second component in \eqref{eq:comppf1} by $L$ and get
  \[
  \begin{split}
    \abs{L\vp}&=\scabs{\int_{\Omega''_{s_1,s_2}} \abs{\nabla v_i}^{p-2} \nabla v_i\cdot \nabla \vp \ud x \ud t}\\
    & \le\|\nabla v_i\|_{L^{q}(\Omega''_{s_1,s_2})}^{p-1} \|\nabla \vp\|_{L^{r'}(\Omega''_{s_1,s_2})}.
  \end{split}
  \]
  Thus $L$ is a bounded functional in the dual space $L^r(s_1,s_2;W^{-1,r}(\Omega''))$, because of the $L^q$-bound for $\nabla v_i$.

  The estimation of the remaining term heavily uses the fact that
  $f_i$ is an $L^1$-function instead of a measure, and can thus be
  estimated on almost every time slice.
   By the Sobolev inequality for $s'>n$, we have
  \begin{align*}
    \left|\int_{\Omega''} f_i(x,t)\varphi(x)\dif x\right|\leq &
    \norm{f_j(\cdot,t)}_{L^1(\Omega'')}\norm{\varphi}_{L^\infty(\Omega'')}\\
    \leq & c\norm{f_i(\cdot,t)}_{L^1(\Omega'')} \norm{\nabla \varphi}_{L^{s'}(\Omega'')}\\
    \leq & c\norm{f_i(\cdot,t)}_{L^1(\Omega'')} \norm{\varphi}_{W^{1,s'}(\Omega'')}.
  \end{align*}
  for all $\varphi\in C^\infty_0(\Omega'')$. Thus
  \begin{displaymath}
    \norm{f_i(\cdot,t)}_{W^{-1,s}(\Omega'')}\leq c \norm{f_i(\cdot,t)}_{L^1(\Omega'')}
  \end{displaymath}
  for almost every $t$, $s_1<t<s_2$, and $1<s<n/(n-1)$ since $C_0^\infty(\Om)$ is dense in $W^{1,s'}_0(\Omega'')$.
  Integrating this estimate in time, we see that
  \begin{displaymath}
    \norm{f_i}_{L^1(s_1,s_2;W^{-1,s}(\Omega''))}\leq c\norm{f_i}_{L^1(\Omega''_{s_1,s_2})}.
  \end{displaymath}
  Combining these estimates, we arrive at
  \begin{displaymath}
    \norm{\partial_t v_i}_{L^1(s_1,s_2;W^{-1,\alpha}(\Omega''))}\leq c,
  \end{displaymath}
  where
  \begin{displaymath}
    1<\alpha<\min\left\{\frac{q}{p-1},\frac{n}{n-1}\right\}.
  \end{displaymath}

 By the estimates in the previous paragraph and Corollary 4 in \cite{Simon}, $(v_i)$ is compact in $L^1(\Omega''_{s_1,s_2})$.
  Thus there is a function $u$ such that $v_i\to u$ in $L^1(\Omega''_{s_1,s_2})$
  for a subsequence.  For the same subsequence we have
  \begin{align*}
    \norm{u-u_i}_{L^1(\Omega''_{s_1,s_2})}\leq & \norm{u-v_i}_{L^1(\Omega''_{s_1,s_2})}
    +\norm{u_i-v_i}_{L^1(\Omega''_{s_1,s_2})}\\
    \leq & \norm{u-v_i}_{L^1(\Omega''_{s_1,s_2})} +\frac ci\to 0 \\
  \end{align*}
  as $i\to \infty$.
  Thus for a  subsequence we get the pointwise almost everywhere convergence of  $u_i$ to $u$ as $i\to\infty$.

  We may pass from convergence in $\Omega''_{s_1,s_2}\Subset \Omega_T$ to the
  full set $\Omega_T$ by the usual exhaustion argument. An application
  of Theorem 5.3 in \cite{KKP} shows that
  $u$ is a supersolution, as well as the pointwise almost everywhere convergence of the gradients.
\end{proof}

\appendix

\section{Approximation in the dual of Sobolev space}

\label{ap:dual-approx}

In this section, we prove Lemmas \ref{lem:meas-approx} and
\ref{lem:para-meas-approx}.
%
We use the characterization of Sobolev spaces on $\R^n$ by Bessel
potentials. For functions $u$ in the Schwartz class $\S$ of rapidly
decreasing functions and $\alpha\geq 0$, define the operator
\begin{displaymath}
  T_\alpha u=g_\alpha\ast u,
\end{displaymath}
where $g_\alpha$ is the Bessel kernel with the property that its Fourier transform
$g_\alpha$ is
\begin{displaymath}
  \Hat{g}_\alpha(\xi)=(1+\abs{\xi}^2)^{-\alpha/2}.
\end{displaymath}
Thus $T_\alpha$ has an inverse $T_{-\alpha}$, with the Fourier multiplier
\begin{displaymath}
  (1+\abs{\xi}^2)^{\alpha/2}.
\end{displaymath}
It is clear that both $T_\alpha$ and $T_{-\alpha}$ map $\S$ into
itself, so we may define both operators on the class of tempered distributions
$\S'$ in the usual manner. Note that the kernel $g_\alpha$ is
symmetric and consequently $T_\alpha$ and $T_{-\alpha}$ are self-adjoint operators.

Let $1<p<\infty$. The Bessel potential space is now defined by
\begin{displaymath}
  \L^{\alpha,p}(\R^n)=\{u\in \S':u=T_\alpha f\text{ for }f\in L^{p}(\R^n)\},
\end{displaymath}
and its dual is
\begin{displaymath}
  \L^{-\alpha,p'}(\R^n)=\{v\in \S':v=T_{-\alpha} g\text{ for }g\in L^{p'}(\R^n)\},
\end{displaymath}
with the duality pairing
\begin{displaymath}
  \inprod{u,v}_{\L^{\alpha,p},\L^{-\alpha,p'}}=\int_{\R^n}fg \dif x,
\end{displaymath}
and the norms
\begin{displaymath}
  \norm{u}_{\L^{\alpha,p}(\R^n)}=\norm{f}_{L^p(\R^n)}\quad\text{and}\quad
  \norm{v}_{\L^{-\alpha,p'}(\R^n)}=\norm{g}_{L^{p'}(\R^n)}.
\end{displaymath}
For $\alpha=0,1,2,\ldots$, we have
\begin{displaymath}
  \L^{\alpha,p}(\R^n)=W^{\alpha,p}(\R^n) \quad\text{and}\quad \L^{-\alpha,p'}(\R^n)
  =W^{-\alpha,p'}(\R^n)
\end{displaymath}
with equivalent norms, see \cite{AH} and \cite{St}.

\begin{proof}[Proof of Lemma \ref{lem:meas-approx}]
  Let $\mu\in W^{-1,p'}(\Omega)$ be a positive measure. We construct a sequence of positive, smooth functions
  $f_i$ such that
  \begin{displaymath}
    \norm{\mu-f_i}_{W^{-1,p'}(\Omega)}\to 0
  \end{displaymath}
  as $i\to\infty$.

  We extend all functions in $W^{1,p}_0(\Omega)$ by zero to $\R^n$, and we also extend $\mu$ to $W^{-1,p}(\R^n)$ by the Hahn-Banach theorem.
  Hence it suffices to find smooth functions $f_i$ defined on $\R^n$ such that
  \begin{displaymath}
    \norm{\mu-f_i}_{W^{-1,p'}(\R^n)}\to 0
  \end{displaymath}
   as $i\to\infty$.
   Through Bessel potentials, we find a function $g\in L^{p'}(\R^n)$ such that
  \begin{displaymath}
    \mu=T_{-1}g.
  \end{displaymath}
  Define
  \begin{displaymath}
    \mu_\varepsilon=\eta_\varepsilon\ast\mu
  \end{displaymath}
  where $\eta_\varepsilon$ is the standard mollifier. Since
  mollification, being defined by convolution, is multiplication on
  the Fourier transform side, and the same is true for $T_{-1}$, we
  have
  \begin{equation}\label{eq:approx-lemma-proof1}
    \mu_\varepsilon=\eta_\varepsilon\ast\mu=\eta_{\varepsilon}\ast(T_{-1}g)
    =T_{-1}(\eta_\varepsilon\ast g).
  \end{equation}
  Thus
  \begin{align*}
    \norm{\mu-\mu_\varepsilon}_{W^{-1,p'}(\R^n)}\leq &
    c\norm{\mu-\mu_\varepsilon}_{\L^{-1,p'}(\R^n)}\\
    = & c\norm{T_{-1}(g-\eta_{\varepsilon}\ast g)}_{\L^{-1,p'}(\R^n)}\\
    = & c\norm{g-\eta_{\varepsilon}\ast g}_{L^{p'}(\R^n)}\to 0
  \end{align*}
  as $\varepsilon\to 0$.
  Finally, we note that since $\mu$ is a positive functional, the functions $\mu_\varepsilon$ are positive,
  by the properties of the standard mollifiers.
\end{proof}

\begin{proof}[Proof of Lemma \ref{lem:para-meas-approx}]
  For the parabolic case, let us denote
  \begin{displaymath}
    V=L^p(\R;W^{1,p}(\R^n))\quad\text{and}\quad V'=L^{p'}(\R;W^{-1,p'}(\R^n)).
  \end{displaymath}
 By the same extensions as before, it suffices to consider these spaces.
 Through Bessel potentials, we see that for $\mu\in V'$, we have
  \begin{displaymath}
    \mu(t)=T_{-1}g_t(x)
  \end{displaymath}
  for almost all $t\in\R$, where $g_t\in L^{p'}(\R^n)$. We define
  \begin{displaymath}
    g(x,t)=g_t(x),
  \end{displaymath}
  for $(x,t)\in \R^{n+1}$, and note that
  \begin{displaymath}
    c^{-1}\norm{g}_{L^{p'}(\R^{n+1})}\leq \norm{\mu}_{V'}\leq c\norm{g}_{L^{p'}(\R^{n+1})}.
  \end{displaymath}

  The approximation is then obtained by defining
  \begin{displaymath}
    \mu_\varepsilon=\eta_\varepsilon\ast\mu,
  \end{displaymath}
  where the standard mollification is taken in $\R^{n+1}$.
  We use the Fourier transform in a fashion similar to
  \eqref{eq:approx-lemma-proof1}, and get
  \begin{displaymath}
    \mu_\varepsilon(t)=T_{-1}[(\eta_\varepsilon\ast g)(t)]
  \end{displaymath}
  for almost all $t$, where $T_{-1}$ depends on the spatial variable only.
  From this we have
  \begin{displaymath}
    \norm{\mu-\mu_\varepsilon}_{V'}\leq c\norm{g-\eta_\varepsilon\ast
      g}_{L^{p'}(\R^{n+1})}\to 0
  \end{displaymath}
as $\varepsilon \to 0$.
 For  positivity, it once again suffices to note that the functions $\mu_\varepsilon$ are positive if the functional $\mu$ is.
\end{proof}

\end{document}